\documentclass[12pt,reqno]{amsart}

\usepackage{amsfonts,amssymb}
\usepackage{amsmath}
\usepackage{amsthm,amscd,latexsym}
\usepackage{color}

\newtheorem{theorem}{Theorem}[section]
\newtheorem{corollary}[theorem]{Corollary}
\newtheorem{lemma}[theorem]{Lemma}
\newtheorem{proposition}[theorem]{Proposition}

\newtheorem{Definition}[theorem]{Definition}

\newtheorem{Example}[theorem]{Example}
\newtheorem{Remark}[theorem]{Remark}

\numberwithin{equation}{section}

\newenvironment{remark}{\begin{Remark}\begin{em}}{\end{em}\end{Remark}}
\newenvironment{example}{\begin{Example}\begin{em}}{\end{em}\end{Example}}

\newcommand{\Pro}{\mathcal P}

\newcommand{\0}{\mathbf{0}}

\def\bP{\mathbb{P}}
\def\cP{\mathcal{P}}

\def\<{\langle}
\def\>{\rangle}
\def\bM{\mathbb{M}}

\def\bR{\mathbb{R}}

\def\diag{\mathrm{diag}}
\def\d{\widehat{d}\,}
\def\<{\langle}
\def\>{\rangle}

\def\Sp{\mathrm{Sp}}

\def\preclog{\underset{\log}\prec}

\DeclareMathOperator{\tr}{tr}
\def\argmin{\mathop{\mathrm{arg\,min}}}
\begin{document}

\title[Eigenvalues  of Cartan barycenter]
{Log-majorizations for the (symplectic) eigenvalues of the Cartan
barycenter}
\author[Hiai and Lim]{Fumio Hiai and Yongdo Lim}
\address{Tohoku University (Emeritus), Hakusan 3-8-16-303, Abiko 270-1154, Japan}
\email{fumio.hiai@gmail.com}
\address{Department of Mathematics, Sungkyunkwan University, Suwon 440-746, Korea}
\email{ylim@skku.edu}
\date{\today}
\maketitle

\begin{abstract} In this paper we show that the
eigenvalue map and the symplectic eigenvalue map of positive definite matrices
are Lipschitz for the Cartan-Hadamard Riemannian  metric, and establish
log-majorizations for the (symplectic) eigenvalues of the Cartan
barycenter of integrable probability Borel measures. This leads a
version of Jensen's inequality for geometric integrals of
matrix-valued integrable  random variables.
\end{abstract}



\section{Introduction}

Let ${\Bbb S}_{n}$ be  the Euclidean space of $n\times n$ real
symmetric matrices equipped with the trace inner product $\langle
X,Y\rangle={\mathrm{tr}}(XY)$. Let ${\Bbb P}_n\subset {\Bbb S}_{n}$
be the open convex cone of real  positive definite matrices, which
is a smooth Riemannian manifold with the Riemannian trace metric
$\<X,Y\>_A=\tr A^{-1}XA^{-1}Y,$ where $A\in {\Bbb P}_{n}$ and
$X,Y\in {\Bbb S}_{n}$. This is an important example of
Cartan-Hadamard manifolds, simply connected complete Riemannian
manifolds with non-positive sectional curvature (the canonical
$2$-tensor is non-negative). The Riemannian distance between
$A,B\in\bP_n$ with respect to the above metric is given by $
\delta(A,B)=\|\log A^{-1/2}BA^{-1/2}\|_2,$ where $\|X\|_2=(\tr
X^2)^{1/2}$ for $X\in{\Bbb S}_n$.


One of recent active research topics on this Riemannian manifold
${\Bbb P}_n$ is the Cartan mean (alternatively the Riemannian mean,
the Karcher mean)
$$G(A_{1},\dots,A_{m}):=\underset{X\in {\Bbb P}}{\argmin}
\sum_{j=1}^{m}\delta^{2} (A_{j}, X),$$ where the minimizer exits
uniquely.  This  is a multivariate extension of the two-variable
geometric mean $A\#B:=A^{1/2}(A^{-1/2}BA^{-1/2})^{1/2}A^{1/2},$
which is the unique midpoint between $A$ and $B$ for the Riemannian
trace metric, and it  retains most of its attractive properties; for
instances, joint homogeneity, monotonicity, joint concavity, and the
arithmetic-geometric-harmonic mean inequalities. It also extends the
multivariate geometric mean on ${\Bbb R}_{+}^n\subset {\Bbb P}_{n}$,
where ${\Bbb R}_+=(0,\infty)$, via the embedding into diagonal
matrices, $(a_{1},\dots,a_{n})\mapsto
{\mathrm{diag}}(a_{1},\dots,a_{n})$.

The Cartan mean extends uniquely to a contractive (with respect to
the Wasserstein metric) barycentric map on the Wasserstein space  of
$L^1$-probability measures; $$G:{\mathcal P}^1({\Bbb P}_{n})\to
{\Bbb P}_{n},$$ where a probability Borel measure $\mu$ belongs to
$\cP^1(\bP_n)$ if $\int_{\bP_n}\delta(A,X)\,d\mu(A)<\infty$ for some
$X\in {\bP}_{n}$. The Cartan barycenter plays a fundamental role in
the theory of integrations (random variables, expectations and
variances). Let $(\Omega,{\bf P})$ be a probability space and let
$L^1(\Omega; {\Bbb P}_{n})$ be the space of measurable functions
$\varphi: \Omega\to {\Bbb P}_{n}$ such that
$\int_{\Omega}\delta(\varphi(\omega),X)\,d{\bf P}(\omega)<\infty$
for some $X\in {\Bbb P}_{n}$. Then the ``geometric" integral of
$\varphi\in L^1(\Omega;{\Bbb P}_{n})$ is naturally defined as
$$\int_{\Omega}^{(G)}\varphi(\omega)\,d{\bf P}(\omega):=G(\varphi_*{\bf P}).$$
Here, we use the notation $\int_\Omega^{(G)}$ to avoid the confusion with
the usual $\int_\Omega$ in the Euclidean (or arithmetic) sense, that
is, $\int_\Omega\varphi(\omega)\,d{\bf
P}(\omega)=\mathcal{A}(\varphi_*{\bf P})$, where ${\mathcal
A}:{\mathcal P}^{\infty}({\Bbb P}_{n})\to {\Bbb P}_{n}$ is the
arithmetic barycenter on the space of bounded probability measures.

In this paper we consider the eigenvalue mapping on ${\Bbb P}_{n}$
$$\lambda:{\Bbb P}_{n}\to {\Bbb R}_{+}^{n}, \ \
\lambda(A)=(\lambda_{1}(A),\dots,\lambda_{n}(A))$$ ordered as
$\lambda_{1}(A) \geq \cdots \geq \lambda_{n}(A)$ counting
multiplicities, and  the \emph{extended symplectic eigenvalue map}
on ${\Bbb P}_{2n}$ \begin{eqnarray}\label{E:seig}\d: {\Bbb
P}_{2n}\to {\Bbb R}_{+}^{2n}, \ \ \
\d(A)=(\d_1(A),\d_2(A),\dots,\d_{2n}(A)).\end{eqnarray} The symplectic
eigenvalues play an important role in classical Hamiltonian
dynamics, in quantum mechanics, in symplectic topology, and in the
more recent subject of quantum information; see, e.g., \cite{Et,Pa}.
For every $A\in\bP_{2n}$, Williamson's theorem (see \cite{ADMS,Pa})
says that there exist a unique diagonal matrix
$D=\diag(d_1,\dots,d_n)$ with $0<d_1\le\dots\le d_n$ and an
$M\in\Sp(2n,\bR),$ the symplectic Lie group, such that $$
A=M^T\begin{bmatrix}D&0\\0&D\end{bmatrix}M.$$ Then,
$d(A)=(d_1(A),\dots,d_n(A)):=(d_1,\dots,d_n)$ is called the {\it
symplectic eigenvalues} of $A$. The  \emph{extended symplectic
eigenvalues} $\d(A)$ of $A$ is defined by
 $$ \d_1(A)=\d_2(A)=d_n,\ \dots\,,\ \
\d_{2n-1}(A)=\d_{2n}(A)=d_1.$$

Our main theorem of the present paper is the following
log-majorizations of the (symplectic) eigenvalues of the Cartan
barycenter.
\begin{theorem}\label{T:main} The maps $\lambda$ and $\d$ are Lipschitz for
the Riemannian trace metric. Moreover,
$$\lambda\left(\int_{\Omega}^{(G)}\varphi(\omega)\,d{\bf P}(\omega)\right)\underset{\log}{\prec}
\int_{\Omega}^{(G)}\lambda(\varphi(\omega))\,d{\bf P}(\omega), \ \ \
\varphi\in L^1(\Omega; {\Bbb P}_{n})$$ and
$$\d\left(\int_{\Omega}^{(G)}\varphi(\omega)\,d{\bf P}(\omega)\right)\underset{\log}{\prec}\int_{\Omega}^{(G)}\d(\varphi(\omega))\,d{\bf P}(\omega),\
\ \ \ \varphi\in L^1(\Omega; {\Bbb P}_{2n}).$$
\end{theorem}

Here $\underset{\log}{\prec}$ denotes the  log-majorization between
positive vectors in $\bR_+^n$; for $a=(a_1,\cdots,a_n)$ and $
b=(b_1,\dots,b_n)$ in $\bR_+^n$ arranged in decreasing order
$a_1\ge\dots\ge a_n$ and $b_1\ge\dots\ge b_n$,
$a\underset{\log}{\prec}b$ if and only if $\prod_{i=1}^{k} a_i \leq
\prod_{i=1}^{k}b_i$ for $1 \leq k \leq n$ and equality holds for
$k=n$. For $A,B\in\bP_n$ we also write $A\underset{\log}{\prec}B$ if
$\lambda(A)\underset{\log}{\prec}\lambda(B)$, which implies that
$||| A|||\leq |||B|||$ for all unitarily invariant norms $|||\cdot
|||$ on the $n\times n$ complex matrices.

The result in the main theorem is a variant of classical Jensen's
inequality for integrals and  covers those of Bhatia and Karandikar
\cite{BK} and of Bhatia and Jain \cite{BJ}:
\begin{align}\label{eigen-log-2}
\lambda(G(A_{1},\dots,A_{m}))&\underset{\log}{\prec}G(\lambda(A_{1}),\dots,\lambda(A_{m}))
\end{align}
 and
\begin{align}\label{eigen-log-3}
\d(G(A_1,\dots,A_m))\underset{\log}{\prec}G(\d(A_{1}),\dots,\d(A_{m})).
\end{align}

\section{(Symplectic) eigenvalue mappings}
The convex cone ${\Bbb P}_{n}$ is, not only a Riemannian manifold
with the Riemannian trace metric, but a Banach Finsler manifold over
${\Bbb S}_n$, the Finsler structure being derived from the operator
norm $\|X\|_{A} := \|A^{-1/2} X A^{-1/2}\|$ for $A\in\bP_n$ and $
X\in{\Bbb S}_n$. The induced metric distance on ${\Bbb P}$ is
explicitly given by $d_T(A,B)=\|\log A^{-1/2}BA^{-1/2}\|$, which is
nothing but the \emph{Thompson metric}
$$
d_T(A,B)={\mathrm{max}}\{\log M(B/A),\log
M(A/B)\},
$$
where $M(B/A) := {\mathrm{inf}}\{\alpha>0: B \leq \alpha
A\}$, the largest eigenvalue of $A^{-1/2}BA^{-1/2}$.
 The geometric mean curve $t\mapsto
A\#_{t}B:= A^{1/2}(A^{-1/2}BA^{-1/2})^{t}A^{1/2}$ is a minimal
geodesic from $A$ to $B$ for the Thompson metric; see \cite{Nu,CPR}.
We observe that
\begin{eqnarray}\label{E:comp}d_T(A,B)\leq \delta(A,B)\leq \sqrt{n}\,d_T(A,B),
\end{eqnarray}
where $\delta(A,B)=\|\log A^{-1/2}BA^{-1/2}\|_2$ is the Riemannian
distance.

Let  $\cP^1(\bP_n)$ be the set of integrable probability Borel
measures on ${\Bbb P}_{n}$, i.e., probability Borel measures $\mu$
on $\bP_n$ such that $\int_{\bP_n}\delta(A,X)\,d\mu(A)<\infty$ for
some $X\in {\Bbb P}_{n}$. By (\ref{E:comp}), the Thompson metric
leads the same probability measure space ${\mathcal P}^{1}({\Bbb
P}_{n})$. That is, for a probability Borel measure $\mu$ on $\bP_n$,
$\int_{\bP_n}\delta(A,X)\,d\mu(A)<\infty$ if and only if
$\int_{\bP_n}d_T(A,X)\,d\mu(A)<\infty$. The \emph{Wasserstein
metric} $\delta^W$  on $\mathcal{P}^{1}({\Bbb P})$ is defined by
$$
\delta^W(\mu,\nu) := \inf_{\pi \in \Pi(\mu,\nu)} \int_{{\Bbb P}_n
\times {\Bbb P}_n} \delta(X,Y) \, d\pi(X,Y),
$$
where $\Pi(\mu,\nu)$ is the set of all couplings for $\mu$ and
$\nu$. Similarly we have the Wasserstein distance  $d_T^W$ from the
Thompson metric. Both are complete metrics on $\mathcal{P}^{1}({\Bbb
P}_n)$ but they are quite distinctive.

For a general metric space $(X,d)$ one can define $\Pro^1(X)$ to be
the set of integrable probability Borel measures whose support has
measure $1$, and the Wasserstein metric $d^W$ on $\Pro^1(X)$ as
above. Then the following result appears in \cite{LL5}.

\begin{lemma}\label{L:Lip}
Let $f:X\to Y$ be a Lipschitz map between complete metric spaces with
Lipschitz constant $C$.  Then $f_*:\mathcal{P}^1(X)\to \mathcal{P}^1(Y)$ is
$d^W$-Lipschitz with Lipschitz constant $C$.
\end{lemma}

Note that if $f:{\Bbb P}_{n}\to {\Bbb P}_{N}$ is a $d_T$-Lipschitz
map with Lipschitz constant $C$, then it is $\delta$-Lipschitz map
with Lipschitz constant $\sqrt{N}C$ by \eqref{E:comp}. It turns out that
the Thompson metric is very useful in studying (sub)homogeneous and monotonic
mappings. A mapping $f:{\Bbb P}_{n}\to {\Bbb P}_{N}$ is said to be
monotonic if $A\leq B$ implies $f(A)\leq f(B)$, and  $f$ is
subhomogeneous of degree $r>0$ if $f(tA)\leq t^{r}f(A)$ for all
$t\geq 1$ and $A\in {\Bbb P}_{n}$.

\begin{proposition}\label{P:main} Let  $f:{\Bbb P}_{n}\to {\Bbb
P}_{N}$  be  monotonic and  subhomogeneous of degree $r$, then it is $d_T$-Lipschitz with
Lipschitz constant $r$.
\end{proposition}

\begin{proof}  Let $A,B>0$ and let  $\alpha=d(A,B)$. Then
$A\leq e^{\alpha}B$ and $B\leq e^{\alpha}A$ by definition of the
Thompson metric. Using monotonicity and subhomogeneity of degree $r>0$, we have
$$f(A)\leq f( e^{\alpha}B)\leq e^{r\alpha}f(B) \ \ {\mathrm{and}}\ \ f(B)\leq
f(e^{\alpha}A)\leq e^{r\alpha}f(A)$$ and hence $d_T(f(A),f(B))\leq
r\alpha=rd_T(A,B)$.
\end{proof}

\begin{example}\label{E:eig}  One can see that the eigenvalue map
$\lambda:{\Bbb P}_{n}\to {\Bbb R}_+^n$ is monotonic and homogeneous
of degree $1$. Indeed, this holds true for the $j$th eigenvalue
mappings
$$\lambda_{i}:{\Bbb P}_{n}\to {\Bbb R}_{+},  \ \ \ i=1,\dots,n.$$
Hence, by Proposition \ref{P:main} and Lemma \ref{L:Lip}, the
push-forward mappings $\lambda_{*}:{\mathcal P}^1({\Bbb P}_{n})\to
{\mathcal P}^1({\Bbb R}_{+}^n)$ and $(\lambda_i)_*:{\mathcal
P}^1({\Bbb P}_{n})\to {\mathcal P}^1({\Bbb R}_{+})$ are
$d_{T}^W$-Lipschitz with Lipschitz constant $1$. By \eqref{E:comp}
they are also $\delta^W$-Lipschitz map with Lipschitz constant
$\sqrt{n}$ and $1,$ respectively.
\end{example}

In fact, the eigenvalue map is also contractive for the Riemannian trace metric $\delta$.

\begin{proposition}\label{P:cont}
The eigenvalue map $\lambda:{\Bbb P}_{n}\to {\Bbb R}_{+}^n$ is $\delta$-contractive$;$
$$\delta(\lambda(A),\lambda(B))\leq \delta(A,B),  \ \ \ \ A,B\in {\Bbb P}_{n}.$$
Moreover,
$\delta^W(\lambda_{*}\mu, \lambda_*\nu)\leq \delta^W(\mu,\nu)$  for $\mu,\nu\in {\mathcal P}^1({\Bbb P}_{n})$.
\end{proposition}

\begin{proof} The first assertion follows from the Lidskii-Wielandt theorem
(see, e.g., \cite{Bh1,Hi}) and the EMI property
(exponential metric increasing property, see \cite{Bh2});
for $A,B\in {\Bbb P}_{n},$
\begin{align*}
\delta(\lambda(A),\lambda(B))&=\|\log \lambda(A)-\log\lambda(B)\|_2=\|\lambda(\log A)-\lambda(\log B)\|_2\\
&\leq\|\log A-\log B\|_2\leq \delta(A,B).
\end{align*}
The latter follows from Lemma \ref{L:Lip}.
\end{proof}

Next, we consider the symplectic eigenvalue map of $2n\times2n$
real positive definite matrices. Let $\bM_{2n}(\bR)$ be the
$2n\times2n$ real matrices and let
$J:=\begin{bmatrix}0&I\\-I&0\end{bmatrix}$ so that $ J^T=J^{-1}=-J.
$ Let $\Sp(2n,\bR)$ denote the group of real {\it symplectic
matrices}, i.e.,
$$
\Sp(2n,\bR):=\{M\in\bM_{2n}(\bR):M^TJM=J\}.
$$

It is straightforward to see that the extended symplectic eigenvalue
mapping (\ref{E:seig})
$$\d:{\Bbb P}_{2n}\to {\Bbb R}_{+}^{2n}$$ is homogeneous of degree $1$. The following
shows that it is monotonic.

\begin{theorem}\label{T:LL} The extended symplectic eigenvalue map $\d$ is monotonic,
i.e., for $A,B\in\bP_{2n}$, $A\le B$ implies $\d(A)\le\d(B)$. Furthermore, for
$A,B\in {\Bbb P}_{2n},$
$$d_{T}(\d(A),\d(B))\leq d_T(A,B) \ \ \ {\mathrm{and}}\ \ \ \delta(\d(A),\d(B))\leq \sqrt{2n}\,\delta(A,B).$$
 \end{theorem}

 \begin{proof}
We first show that
\begin{eqnarray}\label{E:seigen}
\d(A)&=\lambda^{1/2}(A^{1/2}J^TAJA^{1/2}), \ \ A\in {\Bbb P}_{2n}.
\end{eqnarray}

Let $A\in {\Bbb P}_{2n}$. By definition of the symplectic eigenvalues of $A$,
there exist a diagonal matrix  $D=\diag(d_1,\dots,d_n)$ with  $0<d_1\le\dots\le d_n$
and an $M\in\Sp(2n,\bR)$ such that
$
A=M^T\begin{bmatrix}D&0\\0&D\end{bmatrix}M.
$
Set
$$
Q:=\begin{bmatrix}D^{1/2}&0\\0&D^{1/2}\end{bmatrix}MA^{-1/2},
$$
which is a $2n\times2n$ orthogonal matrix as
$$
Q^TQ=A^{-1/2}M^T\begin{bmatrix}D&0\\0&D\end{bmatrix}MA^{-1/2}
=A^{-1/2}AA^{-1/2}=I.
$$
Since $M\in\Sp(2n,\bR)$ implies $M^T\in\Sp(2n,\bR)$ and hence $MJM^T=J$, we have
\begin{align*}
QA^{1/2}JA^{1/2}Q^T
&=\begin{bmatrix}D^{1/2}&0\\0&D^{1/2}\end{bmatrix}MJM^T
\begin{bmatrix}D^{1/2}&0\\0&D^{1/2}\end{bmatrix} \\
&=\begin{bmatrix}D^{1/2}&0\\0&D^{1/2}\end{bmatrix}J
\begin{bmatrix}D^{1/2}&0\\0&D^{1/2}\end{bmatrix}
=\begin{bmatrix}0&D\\-D&0\end{bmatrix}.
\end{align*}
This implies that the eigenvalues of the Hermitian $2n\times2n$ matrix
$A^{1/2}(iJ)A^{1/2}$ is given as
$$
\lambda(A^{1/2}(iJ)A^{1/2})=\lambda\Biggl(\begin{bmatrix}0&iD\\-iD&0\end{bmatrix}\Biggr)
=(d_n,\dots,d_1,-d_1,\dots,-d_n).
$$
Therefore,
\begin{align*}
\lambda^{1/2}(A^{1/2}J^TAJA^{1/2})&=\lambda(|A^{1/2}(iJ)A^{1/2}|) \\
&=(d_n,d_n,d_{n-1},d_{n-1},\dots,d_1,d_1)=\d(A).
\end{align*}

Next, let $A,B\in {\Bbb P}_{2n}$ with $A\leq B$. It follows from \eqref{E:seigen} that
\begin{align*}
\d(A)&=\lambda^{1/2}(A^{1/2}J^TAJA^{1/2})\le\lambda^{1/2}(A^{1/2}J^TBJA^{1/2}) \\
&=\lambda^{1/2}(B^{1/2}JAJ^TB^{1/2})\le\lambda^{1/2}(B^{1/2}JBJ^TB^{1/2})=\d(B).
\end{align*}

The remaining part of proof follows from Proposition \ref{P:main} and (\ref{E:comp}).
\end{proof}

By Theorem \ref{T:LL} and Lemma \ref{L:Lip}, the push-forward map
$\d_*:{\mathcal P}^1({\Bbb P}_{2n})\to {\mathcal P}^1({\Bbb
R}_{+}^{2n})$ is $d_{T}^W$-Lipschitz with Lipschitz constant $1$ and
is also $\delta^W$-Lipschitz with Lipschitz constant $\sqrt{2n}$.
Since $\d_i$ is monotonic and hence is $d_T$-Lipschitz,
$(\d_i)_*:\cP^1(\bP_{2n})\to \cP^1({\Bbb R}_+)$ is $d_T^W$-Lipschitz
by Lemma \ref{L:Lip} again.

\section{Cartan barycenters}
 For $\mu\in\cP^1(\bP_n),$ the
{\it Cartan barycenter} $G(\mu)\in\bP_n$ is  defined as the unique minimizer
$$
G(\mu)=\argmin_{Z\in\bP_m}\int_{\bP_m}\bigl[\delta^2(Z,X)-\delta^2(Y,X)\bigr]\,d\mu(X),
$$
independently of the choice of a fixed $Y\in\bP_n$  (see \cite{St}).
Also, the Cartan barycenter is characterized via the {\it Karcher
equation} (the gradient zero equation) \cite{HL} as
\begin{equation}\label{Karcher}
X=G(\mu)\Longleftrightarrow \int_{{\Bbb P}}\log X^{-1/2}AX^{-1/2}\,d\mu(A)=0.
\end{equation}

An important fact called the \emph{fundamental contraction
property} in \cite{St} (also \cite[Theorem 2.3]{HL}) is that the
Cartan barycenter $G:\cP^1(\bP_n)\to {\bP_n}$ is a Lipschitz map
with Lipschitz constant $1$; namely, for every
$\mu,\nu\in\cP^1(\bP_n)$,
\begin{eqnarray}\label{contract}
\delta(G(\mu),G(\nu))\leq \delta^{W}(\mu,\nu).
\end{eqnarray}
This contraction property also  holds for the Thompson metric
\cite{LL5}.

\begin{example}
In the one-dimensional case on $\bP_1=(0,\infty)={\Bbb R}_{+}$, we find by a direct
computation that for every $\mu\in\cP^1({\Bbb R}_{+})$,
$$
G(\mu)=\exp\int_{{\Bbb R}_{+}}\log x\,d\mu(x).
$$
 Similarly, the Cartan barycenter on  the product space
  ${\Bbb R}_{+}^n$ is given by
$$G(\mu)=\exp\int_{{\Bbb R}_{+}^n}\log x\,d\mu(x), \ \ \ \mu\in \cP^1({\Bbb R}_{+}^n).$$
Here, $\log:{\Bbb R}_{+}^n\to {\Bbb R}^n$ is the usual logarithm
componentwise on the product space ${\Bbb R}_{+}^n$. This coincides
with the restriction of the Cartan barycenter $G:\cP^1({\Bbb
P}_{n})\to {\Bbb P}_{n}$ to $\cP^1({\Bbb D}_{n}),$ where ${\Bbb
D}_{n}$ is the set of all diagonal matrices in ${\Bbb P}_{n}$.
\end{example}

We have an explicit formula of $G(\lambda_{*}\mu)$ for  $\mu\in \cP^1({\Bbb P}_{n});$
\begin{align*}
G(\lambda_{*}\mu)&=\exp\int_{{\Bbb R}_{+}^n}\log x\,d(\lambda_{*}\mu)(x)=\exp\int_{{\Bbb P}_{n}}\log \lambda(A)\,d\mu(A)\\
&=\left(\exp\int_{{\Bbb P}_{n}}\log\lambda_{1}(A)\,d\mu(A),\dots,\exp\int_{{\Bbb P}_{n}}\log\lambda_{n}(A)\,d\mu(A)\right)\\
&=\left(\exp\int_{{\Bbb R}_{+}}\log x\,d(\lambda_1)_{*}\mu(x),\dots,\exp\int_{{\Bbb R}_{+}}\log x\,d(\lambda_n)_{*}\mu(x)\right)\\
&=\left(G(({\lambda_1})_{*}\mu),\dots,G(({\lambda_n})_{*}\mu)\right),
\end{align*}
where in the last equality  the map
 $({\lambda_i})_*:{\mathcal P}^1({\Bbb P}_{n})\to {\mathcal P}^1({\Bbb R}_{+})$ is well-defined by  Example \ref{E:eig}.

Note that for $\mu=(1/m)\sum_{j=1}^m\delta_{A_{j}},$
\begin{eqnarray}\label{EE}
\lambda_*\mu=\frac{1}{m}\sum_{j=1}^{m}\delta_{\lambda(A_{j})} \ \ \ {\mathrm{and}}\ \  \ ({\lambda}_{i})_*\mu
=\frac{1}{m}\sum_{j=1}^{m}\delta_{{\lambda_i}(A_{j})}.
\end{eqnarray}
We have proved the following

\begin{proposition}\label{P:eig}
For $\mu\in \cP^1({\Bbb P}_{n}),$ we have
$$G(\lambda_{*}\mu)=\left(G(({\lambda_1})_{*}\mu),\dots,G(({\lambda_n})_{*}\mu)\right).$$
In particular, for $\mu=(1/m)\sum_{j=1}^m\delta_{A_{j}},$
$$
G(\lambda_{*}\mu)=
G(\lambda(A_{1}),\dots,\lambda(A_{n}))
=\left(\left[\prod_{j=1}^{m}\lambda_{1}(A_{j})\right]^{\frac{1}{m}},\dots,
\left[\prod_{j=1}^{m}\lambda_{n}(A_{j})\right]^{\frac{1}{m}} \right).
$$
\end{proposition}

\section{Log-Majorizations for probability measures}
We have the following diagram involving the eigenvalue map and the Cartan barycenter:
\[ \begin{CD}
{\Bbb P}_{n} @>\lambda>> {\Bbb R}_{+}^n\\
@AG AA     @AAG A\\
{\mathcal P}^{1}({\Bbb P}_{n}) @>{\lambda}_{*}>> {\mathcal P}^{1}({\Bbb R}_{+}^n)
\end{CD}
\]
The diagram does not commute, but finding a relationship between
$\lambda\circ G$ and $G\circ \lambda_*$ seems very interesting. We
establish a log-majorization between them, as well as a similar
log-majorization for the extended symplectic eigenvalues:
\[ \begin{CD}
{\Bbb P}_{2n} @>\d>> {\Bbb R}_{+}^{2n}\\
@AG AA     @AAG A\\
{\mathcal P}^{1}({\Bbb P}_{2n}) @>{\d}_{*}>> {\mathcal P}^{1}({\Bbb R}_{+}^{2n})
\end{CD}
\]

For $0<r<1$ and $\mu\in {\mathcal P}^1({\Bbb P}_{n}),$ let $\mu^r$
denote the push-forward of $\mu$ by the power map $X\mapsto X^{r}$.
Indeed, the power map is a strict contraction for the Riemannian
trace metric (also for the Thompson metric),  as immediately seen
from the log-majorization $\lambda(A^{-r/2}B^rA^{-r/2})
\preclog\lambda^r(A^{-1/2}BA^{-1/2})$, $A,B\in\bP_n$; see
\cite[p.\,229]{Bh2}. Hence the push-forward map $\mu\mapsto\mu^r$ is
a strict contraction from ${\mathcal P}^1({\Bbb P}_{n})$ into
itself.

Let ${\mathcal P}_{0}(\bP_n)$ be the set of all finitely supported
uniform measures on ${\Bbb P}_n$, i.e., measures of the form $\mu =
(1/m) \sum_{j=1}^{m} \delta_{A_{j}}, m\in {\Bbb N},$ where
$\delta_A$ is the point measure of mass $1$ at $A \in {\Bbb P_n}$. We
note that ${\mathcal P}_{0}(\bP_n)$ is dense in the Wasserstein space
${\mathcal P}^1(\bP_n)$ equipped with either $\delta^W$ or $d_{T}^W$.

Let  $\Pro^1({\Bbb S}_{n})$  be the set of probability Borel
measures on the Euclidean space ${\Bbb S}_{n}$ with finite first
moment, i.e., $\int_{{\Bbb S}_{n}}\|X\|_2\,d\mu(X)<\infty$. For each
$\mu\in \Pro^1({\Bbb S}_{n}),$  the identity map on ${\Bbb S}_{n}$
is Bochner $\mu$-integrable and $\mathcal{A}(\mu)=\int_{{\Bbb
S}_{n}}X\, d\mu(X)$ is the arithmetic mean of $\mu.$ Since the
logarithm map $\log: {\Bbb P}_n \to {\Bbb S}_n$ satisfies
$\delta(X,I)=\|\log X\|_{2}$,  the push-forward map $\log_*$ carries
${\mathcal P}^1(\mathbb{P}_n)$ into ${\mathcal P}^1({\mathbb S}_n)$.
In fact,  the EMI property (exponential metric increasing property)
implies that $\log_*:{\mathcal P}^1(\mathbb{P}_n)\to {\mathcal
P}^1({\mathbb S}_n)$ is Lipschitz with Lipschitz constant $1$. This
shows that  the integral $\int_{{\Bbb P}_{n}}\log A\,d\mu(A)\in
{\Bbb S}_{n}$ exists for every $\mu\in {\mathcal P}^1({\Bbb
P}_{n})$. Moreover, similarly to Proposition \ref{P:cont}, the
push-forward $\lambda_*:\Pro^1({\mathbb S}_n)\to\Pro^1(\bR^n)$ of
the eigenvalue map $\lambda:{\mathbb S}_n\to\bR^n$ is Lipschitz with
Lipschitz constant $1$.

\begin{theorem}\label{T:M} We have
\begin{eqnarray}\label{E:int}\lambda(G(\mu))\underset{\log}{\prec}
\lambda^{\frac{1}{r}}(G(\mu^r))\underset{\log}{\prec}
\lambda\left(\exp\int_{{\Bbb P}_{n}}\log
A\,d\mu(A)\right)\underset{\log}{\prec}G(\lambda_{*}\mu)\end{eqnarray}
for every $0<r<1$ and $\mu\in {\mathcal P}^1({\Bbb P}_{n})$.
\end{theorem}

\begin{proof} Let $\mu\in {\mathcal P}^1({\Bbb P}_{n})$. The first log-majorization follows from the  log-majorization of the Cartan barycenter
 appearing in \cite{HL}
$$G(\mu)\underset{\log}{\prec}G(\mu^r)^{\frac{1}{r}}\underset{\log}{\prec}G(\mu^s)^{\frac{1}{s}}, \ \ \ 0<s\leq r<1.$$
As $s\searrow0$  the Lie-Trotter formula \cite{HL1}
\begin{equation*}\label{Lie-Trotter2}
\lim_{s\to0}G(\mu^s)^{\frac{1}{s}}=\exp\int_{\bP_m}\log A\,d\mu(A)
\end{equation*} gives
$$G(\mu)\underset{\log}{\prec}\exp\int_{\bP_m}\log A\,d\mu(A)$$
so that
$$\log \lambda(G(\mu))\underset{}{\prec}\lambda\left(\int_{{\Bbb P}_{m}}\log A\,d\mu(A)\right).$$
For any $\mu\in {\mathcal P}_{0}({\Bbb P}_{n})$, the Ky Fan
majorization (see, e.g., \cite{Bh1,Hi}) yields
$$\lambda\left(\int_{{\Bbb P}_{n}}\log A\,d\mu(A)\right)\underset{}{\prec}\int_{{\Bbb P}_{m}}\lambda(\log A)\,d\mu(A)=
\int_{{\Bbb P}_{n}}\log \lambda(A)\,d\mu(A).$$ As mentioned above
the theorem, note that $\log_*:{\mathcal P}^1(\mathbb{P}_n)\to
{\mathcal P}^1({\mathbb S}_n)$ and $\lambda_*:\Pro^1({\mathbb
S}_n)\to\Pro^1(\bR^n)$ are Lipschitz. Hence, by density of
${\mathcal P}_{0}({\Bbb P}_{n})$ in the Wasserstein space ${\mathcal
P}^1({\Bbb P}_{n})$, the preceding majorization holds for any
$\mu\in {\mathcal P}^1({\Bbb P}_{n})$. Therefore,
$$\lambda\left(\exp\int_{{\Bbb P}_{n}}\log A\,d\mu(A)\right)\underset{\log}{\prec}\exp
\int_{{\Bbb P}_{n}}\log \lambda(A)\,d\mu(A)=G(\lambda_*\mu).$$
\end{proof}

Applying a measure $\mu=(1/m)\sum_{j=1}^{m}\delta_{A_{j}}\in
{\mathcal P}_{0}({\Bbb P}_{n})$ to \eqref{E:int} yields
\begin{align*}
\lambda(G(A_{1},\dots,A_{m}))
&\underset{\log}{\prec}\lambda^{\frac{1}{r}}(G(A_{1}^r,\dots,A_{n}^r))
\underset{\log}{\prec}\lambda\left(\exp\left(\frac{1}{m}\sum_{j=1}^{m}\log A_{j}\right)\right)\\
&\underset{\log}{\prec}G(\lambda(A_{1}),\dots,\lambda(A_{n}))\\
&=\left(\left[\prod_{j=1}^{m}\lambda_{1}(A_{j})\right]^{\frac{1}{m}},
\dots,\left[\prod_{j=1}^{m}\lambda_{n}(A_{j})\right]^{\frac{1}{m}}\right)
\end{align*}
thanks to Proposition \ref{P:eig}.

\begin{remark}
Although we confine ourselves in this paper to the real positive
definite matrices, the results for the eigenvalue map  hold true
when $\bP_n$ is the $n\times n$ complex positive definite matrices.
\end{remark}

Finally we consider the extended symplectic eigenvalue map $\d$.

\begin{theorem}\label{T-4.3}
For every $\mu\in {\mathcal P}^1({\Bbb P}_{2n})$,
\begin{align}\label{F-4.2}
\d^{1\over r}(G(\mu^r))\underset{\log}{\prec}G(\d_{*}\mu),\qquad 0<r\le1.
\end{align}
\end{theorem}

To prove the theorem, we first settle the case where $\mu\in\Pro_0(\bP_{2n})$. For this
we consider slightly more generally the Cartan mean (or the Karcher mean)
$G_w(A_1,\dots,A_m)$ of $A_1,\dots,A_m\in\bP_{2n}$ with a weight $w=(w_1,\dots,w_m)$,
$w_j\ge0$ and $\sum_{j=1}^mw_j=1$.

\begin{lemma}\label{L-4.4}
For every $A_,\dots,A_m\in\bP_{2n}$,
$$
\d^{1\over r}(G_w(A_1^r,\dots,A_m^r))\preclog G_w(\d(A_1),\dots,\d(A_m)),\qquad0<r\le1.
$$
\end{lemma}

\begin{proof}
When $r=1$ this was shown in \cite{BJ}, but the proof below is rather different from that
in \cite{BJ}. First, note that for every $A\in\bP_{2n}$ and $\alpha>0$,
\begin{align}\label{F-4.3}
\d_1(A)\le\alpha\ \iff\ J^TAJ\le\alpha^2 A^{-1}.
\end{align}
Indeed, this is immediately seen from \eqref{E:seigen} since
$$
\lambda^{1/2}(A^{1/2}J^TAJA^{1/2})\le\alpha\ \iff\ J^TAJ\le\alpha^2 A^{-1}.
$$
Now for $j=1,\dots,m$ let $\alpha_j:=\d_1(A_j)$; then
$J^TA_jJ\le\alpha_j^2A_j^{-1}$ by \eqref{F-4.3}. Since $0<r\le1$,
$J^TA_j^rJ\le\alpha_j^{2r}A_j^{-r}$ for $j=1,\dots,m$.
By congruence invariance, monotonicity, joint homogeneity and
self-duality of $G_w$ (see \cite{LL1}) we have
\begin{align*}
J^TG_w(A_1^r,\dots,A_m^r)J&=G_\omega(J^TA_1^rJ,\dots,J^TA_m^rJ)\\
&\le G_w(\alpha_1^{2r}A_1^{-r},\dots,\alpha_m^{2r}A_m^{-r}) \\
&=(\alpha_1^{w_1}\cdots\alpha_m^{w_m})^{2r}G_w(A_1^r,\dots,A_m^r)^{-1},
\end{align*}
which implies by \eqref{F-4.3} again that
$$
\d_1(G_w(A_1^r,\dots,A_m^r))\le(\alpha_1^{w_1}\cdots\alpha_m^{w_m})^r.
$$
Therefore,
$$
\d_1^{1\over r}(G_w(A_1^r,\dots,A_m^r))\le G_w(\d_1(A_1),\dots\d_1(A_m)).
$$
The remaining proof is an application of the standard antisymmetric tensor power
technique (for this see Remark \ref{R-4.5} below), as in the proof of \cite[Theorem 3]{BJ}
with use of \cite[Theorem 4.3]{BK}.
\end{proof}

\begin{remark}\label{R-4.5}
For $k=1,\dots,2n$ let $J^{(k)}:=\wedge^kJ$, the $k$-fold antisymmetric tensor power of $J$.
For any $A\in\bP_{2n}$, since \eqref{E:seigen} implies that
$$
\prod_{i=1}^k\d_i(A)
=\lambda_i^{1/2}\left((\wedge^kA)^{1/2}J^{(k)T}(\wedge^kA)J^{(k)}(\wedge^kA)^{1/2}\right),
$$
the last part of the above proof can be carried out, although $J^{(k)}$ is not a $J$-matrix
of size ${2n\choose k}$ in the definition of the symplectic Lie group
$\Sp({2n\choose k},\bR)$ (see Section 2).
\end{remark}

\noindent
\emph{Proof of Theorem \ref{T-4.3}}.\enspace
Let $0<r\le1$. Lemma \ref{L-4.4} says in particular that \eqref{F-4.2} holds when
$\mu\in\Pro_0(\bP_{2n})$. Now let $\mu\in {\mathcal P}^1({\Bbb P}_{2n})$. By density,
we can find a sequence $\mu_{k}\in {\mathcal P}_{0}({\Bbb P}_{2n})$ converging to
$\mu$ for the Wasserstein metric $\delta^W$. By Theorem \ref{T:LL},
$\delta^W(\d_*\mu_{k},\d_*\mu)\to 0$ as $k\to\infty$. Since $\mu\to\mu^r$ is a contraction
from $\Pro^1(\bP_{2n})$ into itself, $\delta^W(\mu_k^r,\mu^r)\le\delta^W(\mu_k,\mu)\to0$.
By the fundamental contraction property,
$$\delta(G(\mu_{k}^r),G(\mu^r))\leq \delta^W(\mu_{k}^r,\mu^r)\to 0$$
and also
$$\delta(G(\d_*\mu_{k}),G(\d_*\mu))\leq \delta^W(\d_*\mu_{k},\d_*\mu)\to 0.$$
Since $\d$ and $\d_*$ are  continuous, we have $\d(G(\mu_k^r))\to
\d(G(\mu^r))$ as well as $G(\d_*\mu_{k})\to G(\d_*\mu)$ in ${\Bbb
R}_{+}^{2n}$. By Lemma \ref{L-4.4} we have $\d^{1\over
r}(G(\mu_k^r))\underset{\log}{\prec}G(\d_*\mu_k)$. Hence letting
$k\to \infty$ gives $\d^{1\over
r}(G(\mu^r))\underset{\log}{\prec}G(\d_{*}\mu)$.\qed

\begin{remark}
Let $0<r<1$. Compared with the log-majorizations in \eqref{T:M} one may think of the
following, where $\mu\in\Pro^1(\bP_{2n})$, $A,B\in\bP_{2n}$ and $0<t<1$:
\begin{itemize}
\item[(a)] $\d(G(\mu^r)^{1\over r})\preclog G(\d_*\mu)$? In particular,
$\d((A^r\#_t B^r)^{1\over r})\preclog\d^{1-t}(A)\d^t(B)$?
\item[(b)] $\d(G(\mu)^r)\preclog\d(G(\mu^r))$? In particular,
$\d((A\#_t B)^r)\preclog\d(A^r\#_t B^r)$?
\item[(c)] $\d\left(\exp\int_{\bP_{2n}}\log X\,d\mu(X)\right)\preclog G(\d_*\mu)$?
\end{itemize}
When $n=1$, since $\d(X)=(\det^{1\over2}(X),\det^{1\over2}(X))$ for any $X\in\bP_2$, the
above are all trivial as both sides of each of (a)--(c) are equal. However, when $n\ge2$,
it is rather difficult for us to expect that the log-majorizations in (a)--(c) hold true,
while we have no explicit counterexamples.
\end{remark}

We have directly the following general version, which provides the
proof of the main result (Theorem \ref{T:main}). Indeed,
$\varphi_*{\bf P}\in {\mathcal P}^1({\Bbb P}_{n})$  for every
$\varphi\in L^1(\Omega;{\Bbb P}_{n}),$ where $(\Omega, {\bf P})$ is
a probability space, and then by Theorem \ref{T:M},
$$\lambda(G(\varphi_*{\bf P}))\underset{\log}{\prec}G(\lambda_*(\varphi_*{\bf
P}))=G((\lambda\circ \varphi)_*{\bf P}),$$ and similarly for the
case of the symplectic eigenvalues when $\varphi\in L^1(\Omega;{\Bbb
P}_{2n})$.

\begin{theorem} Let $(\Omega, {\bf P})$ be a probability space. Then for every $\varphi\in L^1(\Omega;{\Bbb P}_{n}),$ that is,
 $\varphi:\Omega\to {\Bbb P}_{n}$ satisfying $\int_{\Omega}\delta(\varphi(\omega),X)\,d{\bf P}(\omega)<\infty$ for some $X\in {\Bbb P}_{n},$
\begin{eqnarray}\label{integ}\lambda(G(\varphi_*{\bf P}))\underset{\log}{\prec}G((\lambda\circ \varphi)_*{\bf P}).
\end{eqnarray} Moreover, for every $\varphi\in L^1(\Omega;{\Bbb P}_{2n}),$
\begin{eqnarray}\label{integ}\d(G(\varphi_*{\bf P}))\underset{\log}{\prec}G((\d\circ \varphi)_*{\bf P}).
\end{eqnarray}
\end{theorem}

More precisely we have from  (\ref{E:int}),

\begin{corollary} For every $\varphi\in L^1(\Omega; {\Bbb P}_{n}),$
\begin{align*}
\lambda\left(\int_{\Omega}^{(G)}\varphi(\omega)\,d{\bf
P}(\omega)\right)&\underset{\log}{\prec}
\lambda^{\frac{1}{r}}\left(\int_{\Omega}^{(G)}\varphi(\omega)^r\,d{\bf
P}(\omega)\right)\\
&\underset{\log}{\prec}
\lambda\left(\exp\int_{\Omega}\log
\varphi(\omega)\,d{\bf P}(\omega)\right)\\
&\underset{\log}{\prec}
\int_{\Omega}^{(G)}\lambda(\varphi(\omega))\,d{\bf P}(\omega).
\end{align*}
\end{corollary}


\section{Log-majorizations for multiple probability measures}
There is a natural notion of multivariate ``geometric" mean of
integrable probability Borel measures \cite{HLL}. The Cartan mean of
$m$ positive definite matrices
 $G:{\Bbb P}_{n}^m\to {\Bbb P}_{n}$  is Lipschitz from the fundamental contraction property and hence induces   a Lipschitz map
 $\Lambda:({\mathcal P}^1({\Bbb P}_{n}))^m\to {\mathcal P}^1({\Bbb P}_{n})$ defined by
 $$\Lambda(\mu_{1},\dots,\mu_{m}):=G_{*}(\mu_{1}\times\cdots\times \mu_{m})\in {\mathcal P}^1({\Bbb P}_{n}).$$
Note that  $\Lambda(\mu)=\mu$ for $m=1$.
 By our log-majorization
in Theorem \ref{T:M},
\begin{align}\label{F-5.1}
\lambda (G(\Lambda(\mu_{1},\dots,\mu_{m})))\underset{\log}{\prec}G(\lambda_{*}\Lambda(\mu_{1},\dots,\mu_{m}))=
G((\lambda\circ G)_{*}(\mu_{1}\times\cdots\times\mu_{m})).
\end{align}
However, from  $\lambda_*\mu_{j}\in {\mathcal P}^1({\Bbb R}_{+}^n),$
$$\Lambda(\lambda_*\mu_{1},\dots,\lambda_*\mu_{m}):=G_*(\lambda_*\mu_{1}\times \cdots\times
\lambda_*\mu_{n})\in {\mathcal P}^1({\Bbb R}_{+}^n)$$ and
$G(\Lambda(\lambda_*\mu_{1},\dots,\lambda_*\mu_{m}))\in {\Bbb
R}_{+}^n$. Between this and both sides of \eqref{F-5.1} we have the
following log-majorizations.

\begin{theorem}\label{P-5.1}
For every $\mu_1,\dots,\mu_m\in\Pro^1(\bP_n)$,
\begin{align}\label{F-5.2}
\lambda(G(\Lambda(\mu_1,\dots,\mu_m)))\preclog G(\lambda_*\Lambda(\mu_1,\dots,\mu_m))
\preclog G(\Lambda(\lambda_*\mu_1,\dots,\lambda_*\mu_m)).
\end{align}
\end{theorem}

\begin{proof}
It remains to prove the second log-majorization. As mentioned above the theorem, note that
$G:\bP_n^m\to\bP_n$ and $\Lambda:(\Pro^1(\bP_n))^m\to\Pro^1(\bP_n)$ are
Lipschitz continuous, as well as so are $\lambda:\bP_n\to\bR_+^n$ and
$\lambda_*:\Pro^1(\bP_n)\to\Pro^1(\bR_+^n)$ (see Example \ref{E:eig}). So it suffices by
continuity to prove the assertion for $\mu_1,\dots,\mu_n\in\Pro_0(\bP_n)$. Let
$\mu_j=(1/k_j)\sum_{i=1}^{k_j}\delta_{A_{ji}}$ for $j=1,\dots,m$. Then
$$
\Lambda(\mu_1,\dots,\mu_m)
={1\over k_1\cdots k_m}\sum_{i_1,\dots,i_m}\delta_{G(A_{1i_1},\dots,A_{mi_m})},
$$
where the sum is taken over all $i_j=1,\dots,k_j$ and $j=1,\dots,m$.
We hence have from (\ref{EE})
$$
\lambda_*\Lambda(\mu_1,\dots,\mu_m)
={1\over k_1\cdots k_m}\sum_{i_1,\dots,i_m}\delta_{\lambda(G(A_{1i_1},\dots,A_{mi_m}))}
$$
so that
\begin{align}\label{F-5.3}
G(\lambda_*\Lambda(\mu_1,\dots,\mu_m))
=G\bigl(\lambda(G(A_{1i_1},\dots,A_{mi_m})):i_1,\dots,i_m\bigr),
\end{align}
where the right-hand side of \eqref{F-5.3} is the geometric mean as an element of
$(\bR_+^n)^{k_1\cdots k_m}$. On the other hand, since
$\lambda_*\mu_j=(1/k_j)\sum_{i=1}^{k_j}\delta_{\lambda(A_{ji})}$, we have
\begin{align}\label{F-5.4}
G(\Lambda(\lambda_*\mu_1,\dots,\lambda_*\mu_m))
=G\bigl(G(\lambda(A_{1i_1}),\dots,\lambda(A_{mi_m})):i_1,\dots,i_m\bigr).
\end{align}
By the log-majorization in \cite[(30)]{BK} (also Theorem \ref{T:M}),
$$
\lambda(G(A_{1i_1},\dots,A_{mi_m}))\preclog
G(\lambda(A_{1i_1}),\dots,\lambda(A_{mi_m}))
$$
for all $i_1,\dots,i_m$. Combining this with \eqref{F-5.3} and \eqref{F-5.4} we easily see
the second log-majorization asserted.
\end{proof}

When $m=1$, since $\lambda(G(\Lambda(\mu)))=\lambda(G(\mu))$ and
$G(\lambda_*\Lambda(\mu))=G(\Lambda(\lambda_*\mu))=G(\lambda_*\mu)$,
\eqref{F-5.2} is included in \eqref{E:int}. When
$\mu_j=\delta_{A_j}$ for $j=1,\dots,m$, since the first two terms of
\eqref{F-5.2} are $\lambda(G(A_1,\dots,A_m))$ from
$\Lambda(\delta_{A_{1}},\dots,\delta_{A_{m}})=\delta_{G(A_{1},\dots,A_{m})}$
and the last term is $G(\lambda(A_1),\dots,\lambda(A_m))$
by (\ref{EE}), \eqref{F-5.2} reduces to \eqref{eigen-log-2}.

For $\mu_1,\dots,\mu_m\in\Pro^1(\bP_{2n})$ the log-majorization in Theorem \ref{T-4.3} gives
$$
\d(G(\Lambda(\mu_1,\dots,\mu_m)))\preclog
G(\d_*\Lambda(\mu_1,\dots,\mu_m)).
$$
The proof of the second log-majorization of \eqref{F-5.5} is similar to that of \eqref{F-5.2}
above by using \cite[(20)]{BJ} (also Theorem \ref{T-4.3}) in place of \cite[(30)]{BK}.

\begin{theorem}\label{P-5.2}
For every $\mu_1,\dots,\mu_m\in\Pro^1(\bP_{2n})$,
\begin{align}\label{F-5.5}
\d(G(\Lambda(\mu_1,\dots,\mu_m)))\preclog G(\d_*\Lambda(\mu_1,\dots,\mu_m))
\preclog G(\Lambda(\d_*\mu_1,\dots,\d_*\mu_m)).
\end{align}
\end{theorem}

\section{Acknowledgements}

The work of F.~Hiai was supported in part by Grant-in-Aid for
Scientific Research (C)17K05266. The work of Y. Lim was supported by
the National Research Foundation of Korea (NRF) grant funded by the
Korea government (MEST) No.NRF-2015R1A3A2031159 and
2016R1A5A1008055.


\end{document}